\documentclass[a4paper,intlimits,reqno,10pt,oneside]{amsart}
\usepackage{amssymb}

\usepackage[centertags]{amsmath}
\usepackage{amsfonts}
\usepackage{amssymb}
\usepackage{amsthm}

\newcommand{\Natural}{\mathbb N}

\newcommand{\Real}{\mathbb R}

\newcommand{\set}[1]{\left\{#1\right\}}

\newcommand{\diam}{\mathop{\mathrm{diam}}\nolimits}
\newcommand{\dist}{\mathop{\mathrm{dist}}\nolimits}
\newcommand{\rng}{\mathop{\mathrm{rng}}\nolimits}
\newcommand{\dent}{\mathop{\mathrm{Dz}}\nolimits}
\newcommand{\szlenk}{\mathop{\mathrm{Sz}}\nolimits}

\newcommand{\norm}[1]{\left\Vert#1\right\Vert}
\newcommand{\duality}[1]{\left\langle#1\right\rangle}
\newcommand{\ws}{weak$^*$}

\def\cal{\mathcal}
\def\eps{\varepsilon}
\theoremstyle{plain}
\newtheorem{thm}{Theorem}
\newtheorem{cor}[thm]{Corollary}
\newtheorem{lem}[thm]{Lemma}
\newtheorem{prop}[thm]{Proposition}

\theoremstyle{definition}
\newtheorem{defn}[thm]{Definition}

\begin{document}

\title{Weak$^*$ dentability index of spaces $C([0,\alpha])$}
\author{Petr H\'ajek}
\address{Mathematical Institute\\Czech Academy of Science\\\v Zitn\'a 25\\115 67 Praha 1\\Czech Republic}
\email{hajek@math.cas.cz}
\author{Gilles Lancien}
\address{Universit\' e de Franche Comt\' e, Besan\c con\\ 16, Route de Gray\\ 25030 Besan\c con Cedex\\ France}
\email{gilles.lancien@univ-fcomte.fr}
\author{Anton\'\i n Proch\'azka}
\address{Charles University\\Sokolovsk\'a 83\\186 75 Praha 8\\Czech Republic
and  Universit\'e Bordeaux 1, 351 cours de
la liberation, 33405, Talence, France.} \email{protony@karlin.mff.cuni.cz}

\date{July 2008}
\thanks{Supported by grants: Institutional Research Plan AV0Z10190503,  A100190801, GA \v CR 201/07/0394}

\begin{abstract}
We compute the weak$^*$-dentability index of the spaces $C(K)$ where $K$ is a
countable compact space. Namely $\dent(C([0,\omega^{\omega^\alpha}])) =
\omega^{1+\alpha+1}$, whenever $0\le\alpha<\omega_1$. More generally, $\dent(C(K))=\omega^{1+\alpha+1}$ if $K$ is a scattered compact whose height $\eta(K)$ satisfies $\omega^\alpha<\eta(K)\leq \omega^{\alpha+1}$ with an $\alpha$ countable.
\end{abstract}

\maketitle

\section{Introduction}
The Szlenk index has been introduced in \cite{Sz} in order to show that there
is no universal space for the class of separable reflexive Banach spaces. The
general idea of assigning an isomorphically invariant ordinal index to a class
of Banach spaces proved to be extremely fruitful in many situations. We refer
to \cite{O} for a survey with references. In the present note we will give an
alternative geometrical description of the Szlenk index  (equivalent to the
original definition whenever $X$ is a separable Banach space not containing any
isomorphic copy of $\ell_1$ \cite{La0}), which stresses its close relation to
the weak$^*$-dentability index. The later index proved to be very useful in
renorming theory (\cite{La0}, \cite{La1}, \cite{La2}).

Let us proceed by giving the precise definitions. Consider a real Banach space
$X$ and $K$ a weak$^*$-compact subset of $X^*$. For $\eps>0$ we let $\cal V$ be
the set of all relatively weak$^*$-open subsets $V$ of $K$ such that the norm
diameter of $V$ is less than $\eps$ and $s_{\eps}K=K\setminus \bigcup\{V:V\in\cal
V\}.$ Then we define inductively $s_{\eps}^{\alpha}K$ for any ordinal $\alpha$
by $s^{\alpha+1}_{\eps}K=s_{\eps}(s_{\eps}^{\alpha}K)$ and
$s^{\alpha}_{\eps}K={ \bigcap_{\beta<\alpha}}s_{\eps}^{\beta}K$ if
$\alpha$ is a limit ordinal. We denote by $B_{X^*}$ the closed unit ball of
$X^*$. We then define $\text{Sz}(X,\eps)$ to be the least ordinal $\alpha$ so
that $s_{\eps}^{\alpha}B_{X^*}=\emptyset,$ if such an ordinal exists. Otherwise
we write $\text{Sz}(X,\eps)=\infty.$ The {\it Szlenk index} of $X$ is finally
defined by $\text{Sz}(X)=\sup_{\eps>0}\text{Sz}(X,\eps)$. Next, we introduce
the notion of {\it weak$^*$-dentability index}. Denote $H(x,t)=\{ x^*\in K,\
x^*(x)>t\},$ where $x\in X$ and $t\in \mathbb{R}$. Let $K$ be again a
weak$^*$-compact. We introduce a weak$^*$-slice of $K$ to be any non empty set
of the form $H(x,t)\cap K$ where $x\in X$ and $t\in \mathbb{R}$.
 Then we denote by $\cal
S$ the set of all weak$^*$-slices of $K$ of norm diameter less than $\eps$ and
$d_{\eps}K=K\setminus \bigcup\{S:S\in\cal S\}.$ From this derivation, we define
inductively $d_{\eps}^{\alpha}K$ for any ordinal $\alpha$ by
$d^{\alpha+1}_{\eps}K=s_{\eps}(d_{\eps}^{\alpha}K)$ and
$d^{\alpha}_{\eps}K={ \bigcap_{\beta<\alpha}}s_{\eps}^{\beta}K$ if
$\alpha$ is a limit ordinal. We then define $\dent(X,\eps)$ to be the least
ordinal $\alpha$ so that $d_{\eps}^{\alpha}B_{X^*}=\emptyset,$ if such an
ordinal exists. Otherwise we write $\dent(X,\eps)=\infty.$ The
weak$^*$-dentability index is defined by
$\dent(X)=\sup_{\eps>0}\dent(X,\eps).$

Let us now recall that it follows from the classical theory of Asplund spaces
(see for instance \cite{HMVZ}, \cite{HLM}, \cite{DGZ} and references therein)
that for a Banach space $X$, each of the following conditions:
$\text{Dz}(X)\neq \infty$ and $\text{Sz}(X)\neq \infty$ is equivalent to $X$
being an Asplund space. In particular, if $X$ is a separable Banach space, each
of the conditions $\text{Dz}(X)<\omega_1$ and $\text{Sz}(X)<\omega_1$ is
equivalent to the separability of $X^*$. In other words, both of these indices
measure ``quantitatively" the ``Asplundness" of the space in question.
Moreover, these indices are invariant under isomorphism.

It is immediate from the definition, that $\text{Dz}(X)\ge\text{Sz}(X)$ for
every Banach space $X$. Relying on tools from descriptive set theory, Bossard
(for the separable case, see \cite{Bo1} and \cite{Bo2}) and the second named
author (\cite{La2}), proved non-constructively that there exists a universal
function $\psi:\omega_1 \to \omega_1$, such that if $X$ is an Asplund space
with $\text{Sz}(X)<\omega_1$, then $\text{Dz}(X)\leq \psi(\text{Sz}(X))$.

Recently, Raja~\cite{R} has obtained a concrete example of such a $\psi$, by
showing that $\dent(X) \leq \omega^{\szlenk(X)}$ for every Asplund space. This
is a very satisfactory result, but it is not optimal, as we know from \cite{HL}
that the optimal value $\psi(\omega)=\omega^2$. Further progress in this area
depends on the exact knowledge of indices for concrete spaces. The Szlenk index
has been precisely calculated for several classes of spaces, most notably for
the class of $C([0,\alpha])$, $\alpha$ countable (Samuel \cite{Sa}, see also
\cite{HL}). We have $\szlenk(C([0,\omega^{\omega^\alpha}])) =
\omega^{\alpha+1}$, so it follows from the Bessaga-Pe\l czy\'nski (\cite{BP})
Theorem 1 below, that the value
 of the Szlenk index characterizes the isomorphism class (\cite{HMVZ}).
Computations of the Szlenk index for other spaces may be found e.g. in
\cite{AJO}, \cite{Al}, \cite{KOS}. On the other hand, the precise value of the
weak$^*$-dentability index is known only for superreflexive Banach spaces,
where $\text{Dz}(X)=\omega$ (\cite{La1}, \cite{HMVZ}), and for spaces with an
equivalent UKK$^*$ renorming (\cite{HL}). For a detailed background information
on the Szlenk and dentability indices we refer the  reader to \cite{HMVZ},
\cite{L}, \cite{O}, \cite{Ro} and references therein.

The main result of our note, Theorem~\ref{t:main}, is a precise evaluation of
the $w^*$-dentability index for the class of $C([0,\alpha])$, $\alpha$
countable. These spaces  have been classified isomorphically by C. Bessaga and
A. Pe\l czy\'nski \cite{BP} in the following way.

\begin{thm}\label{l:class}(Bessaga-Pe\l czy\'nski)
Let $\omega \leq \alpha \leq \beta < \omega_1$. Then $C([0,\alpha])$ is
isomorphic to $C([0,\beta])$ if and only if $\beta < \alpha^\omega$. Moreover,
for every countable compact space $K$ there exists a unique $\alpha<\omega_1$
such that $C(K)$ is isomorphic to $C([0,\omega^{\omega^\alpha}])$.
\end{thm}

It is also well-known and easy to show that for $\alpha\ge \omega$,
$C([0,\alpha])$ is isomorphic to $C_0([0,\alpha])$ where
$C_0([0,\alpha])=\set{f\in C([0,\alpha]):f(\alpha)=0}$. The aim of
this note is to prove the next theorem. Note, as a particular
consequence, that the weak$^*$-dentability index gives a complete
isomorphic characterization of a $C(K)$ space, when $K$ is a
metrizable compact space (similarly to the case of the Szlenk
index).

\begin{thm}\label{t:main}
Let $0 \leq \alpha <\omega_1$. Then $\dent(C([0,\omega^{\omega^\alpha}])) =
\omega^{1+\alpha+1}$.
\end{thm}

\begin{proof}
We start by proving  the upper estimate

\begin{equation}\label{i:inequality1}
\dent(C([0,\omega^{\omega^\alpha}])) \leq \omega^{1+\alpha+1},
\end{equation}

The method of the proof is similar to \cite{HL}, where a short and direct
computation of the Szlenk index of the spaces $C([0,\alpha])$ is presented.
Next lemma is a variant of Lemma 2.2. from~\cite{HL}. We omit the proof which
requires only minor notational changes.

\begin{lem}\label{l:homog}
Let $X$ be a Banach space and $\alpha$ an ordinal. Assume that
\[
\forall \varepsilon>0 \quad \exists\delta(\varepsilon)>0 \quad
d_\varepsilon^\alpha(B_{X^*}) \subset (1-\delta(\varepsilon))B_{X^*}.
\]
Then 
\[
\dent(X) \leq \alpha\cdot\omega.
\]
\end{lem}

We shall also use the following Lemma that can be found in \cite{L}.

\begin{lem}\label{l:szlenkL2} Let $X$ be a Banach space and $L_2(X)$ be the Bochner
space $L_2([0,1],X)$. Then
\[
\dent(X) \le \szlenk(L_2(X)).
\]
\end{lem}

Thus, in order to obtain the desired upper bound we only need to
prove the following.

\begin{prop}\label{p:upperszlenkL2} Let $0\le \alpha <\omega_1$. Then
$\szlenk(L_2(C([0,\omega^{\omega^\alpha}])))\le \omega^{1+\alpha+1}$.
\end{prop}

\begin{proof} For a fixed $\alpha<\omega_1$ and $\gamma<
\omega^{\omega^\alpha}$, let us put
$Z=L_2(\ell_1([0,\omega^{\omega^\alpha})))$, together with the
\ws-topology induced by $L_2(C_0([0,\omega^{\omega^\alpha}]))$ and
$Z_\gamma=L_2(\ell_1([0,\gamma]))$ with the \ws-topology induced by
$L_2(C([0,\gamma]))$. We recall that for a Banach space $X$ with
separable dual, $L_2(X^*)$ is canonically isometric to $(L_2(X))^*$.

Let $P_\gamma$ be the canonical projection from
$\ell_1([0,\omega^{\omega^\alpha}))$ onto $\ell_1([0,\gamma])$.
Then, for $f\in Z$ and $t\in [0,1]$, we define $(\Pi_\gamma
f)(t)=P_\gamma(f(t))$. Clearly, $\Pi_\gamma$ is a norm one
projection from $Z$ onto $Z_\gamma$ (viewed as a subspace of $Z$).
We also have that for any $f\in Z$, $\|\Pi_\gamma f -f\|$ tends to
$0$ as $\gamma$ tends to $\omega^{\omega^\alpha}$.

Next is a variant of Lemma 3.3 in~\cite{HL}.

\begin{lem}\label{l:subspaces}
Let $\alpha<\omega_1$, $\gamma< \omega^{\omega^\alpha}$, $\beta <
\omega_1$ and $\varepsilon>0$. If $z \in
s_{3\varepsilon}^\beta(B_{Z})$ and $\norm{\Pi_\gamma
z}^2>1-\varepsilon^2$, then $\Pi_\gamma z \in
s_\varepsilon^\beta(B_{Z_\gamma})$.
\end{lem}

\begin{proof}
We will proceed by transfinite induction in $\beta$. The cases
$\beta=0$ and $\beta$ a limit ordinal are clear. Next we assume that
$\beta=\mu+1$ and the statement has been proved for all ordinals
less than or equal to $\mu$.
 Consider
$f \in B_Z$ with $\norm{\Pi_\gamma f}^2>1-\varepsilon^2$ and
$\Pi_\gamma f \notin s_\varepsilon^\beta(B_{Z_\gamma})$. Assuming $f
\notin s_{3\varepsilon}^\mu(B_{Z}) \supset
s_{3\varepsilon}^\beta(B_{Z})$ finishes the proof, so we may suppose
that $f \in s_{3\varepsilon}^\mu(B_{Z})$. By the inductive
hypothesis, $\Pi_\gamma f \in s_{\varepsilon}^\mu(B_{Z_\gamma})$.
Thus there exists a \ws-neighborhood $V$ of $f$ such that the
diameter of $V\cap s_{\varepsilon}^\mu(B_{Z_\gamma})$ is less than
$\varepsilon$. We may assume that $V$ can be written
$V=\bigcap_{i=1}^k H(\varphi_i,a_i)$, where $a_i\in \mathbb{R}$ and
$\varphi_i\in L_2(C([0,\gamma]))$. We may also assume, using
Hahn-Banach theorem, that $V\cap
(1-\varepsilon^2)^{1/2}B_{Z_\gamma}=\emptyset$.

Define $\Phi_i \in L_2(C_0([0,\omega^{\omega^\alpha}))$ by
$\Phi_i(t)(\sigma)=\varphi_i(t)(\sigma)$ if $\sigma\le \gamma$ and
$\Phi_i(t)(\sigma)=0$ otherwise. Then define $W=\bigcap_{i=1}^k
H(\Phi_i,a_i)$. Note that for $f$ in $Z$, $f\in W$ if and only if
$\Pi_\gamma f \in V$. In particular $W$ is a \ws-neighborhood of
$f$. Consider now $g,g' \in W\cap s_{3\varepsilon}^\mu(B_Z)$. Then
$\Pi_\gamma g$ and $\Pi_\gamma g'$ belong to $V$ and therefore they have norms greater than $(1-\varepsilon^2)^{1/2}$. It follows from the
induction hypothesis that $\Pi_\gamma g, \Pi_\gamma g' \in s^\mu_\varepsilon(B_{Z_\gamma})$ thus $\|\Pi_\gamma g -\Pi_\gamma g'\| \le
\varepsilon$. Since $\|\Pi_\gamma g\|^2>1-\varepsilon^2$ and
$\|g\|\le 1$, we also have $\|g-\Pi_\gamma g\| <\varepsilon$. The
same is true for $g'$ and therefore $\|g-g'\|<3\varepsilon$. This
finishes the proof of the Lemma.
\end{proof}

We are now in position to prove Proposition \ref{p:upperszlenkL2}.
For that purpose it is enough to show that for all
$\alpha<\omega_1$:

\begin{equation}\label{e:enough}
\forall \gamma<\omega^{\omega^\alpha} \quad \forall \varepsilon>0
\quad s_\varepsilon^{\omega^{1+\alpha}}(B_{Z_\gamma})= \emptyset.
\end{equation}
We will prove this by transfinite induction on $\alpha<\omega_1$.

For $\alpha=0$, $\gamma$ is finite and the space
$Z_\gamma$ is isomorphic to $L_2$ and therefore
$s_\varepsilon^\omega(B_{Z_\gamma})$ is empty. So \eqref{e:enough}
is true for $\alpha=0$.

Assume that \eqref{e:enough} holds for $\alpha<\omega_1$.
Let $Z=L_2(C_0([0,\omega^{\omega^\alpha}]))$. It follows from
Lemma~\ref{l:subspaces} and the fact that for all $f\in Z$
$\|\Pi_\gamma f -f\|$ tends to $0$ as $\gamma$ tends to
$\omega^{\omega^\alpha}$, that
\[
\forall \varepsilon >0 \quad s_\varepsilon^{\omega^{1+\alpha}}(B_Z)
\subset (1-\varepsilon^2)^{1/2}B_Z.
\]
From this and Lemma~\ref{l:homog} it follows that
\[
\forall \varepsilon >0 \quad
s_\varepsilon^{\omega^{1+\alpha+1}}(B_Z)=\emptyset.
\]
By Theorem~\ref{l:class} we know that the spaces $C([0,\gamma])$,
$C([0,\omega^{\omega^\alpha}])$, and also
$C_0([0,\omega^{\omega^\alpha}])$ are isomorphic, whenever
$\omega^{\omega^\alpha} \leq \gamma < \omega^{\omega^{\alpha+1}}$.
Thus $s_\varepsilon^{\omega^{1+\alpha+1}}(B_{Z_\gamma})=\emptyset$
for any $\varepsilon>0$ and $\gamma < \omega^{\omega^{\alpha+1}}$,
i.e. \eqref{e:enough} holds for  $\alpha+1$.

Finally, the induction is clear for limit ordinals.
\end{proof}

In the rest of the note, we will focus on proving the converse inequality. Note
that it suffices to deal with the spaces $C([0,\omega^{\omega^\alpha}])$ where
$\alpha<\omega$. Indeed, in case $\alpha\ge\omega$, our inequality
~\eqref{i:inequality1} implies that

$$
\dent(C([0,\omega^{\omega^\alpha}]))=\szlenk(C([0,\omega^{\omega^\alpha}])) =
\omega^{\alpha+1}.
$$

\begin{prop}\label{proposition:fundamental}
Let $X,Z$ be Banach spaces and let $Y \subset X^*$ be a closed subspace. Let
there be $T \in {\mathcal B}(X,Z)$ such that $T^*$ is an isometric isomorphism
from $Z^*$ onto $Y$.  Let $\varepsilon >0$, $\alpha$ be an ordinal such that
$B_{X^*}\cap Y \subset d_\varepsilon^\alpha(B_{X^*})$, and $z \in Z^*$. If $z
\in d_\varepsilon^\beta(B_{Z^*})$, then $T^*z \in
d_\varepsilon^{\alpha+\beta}(B_{X^*})$.
\end{prop}

\begin{proof}
By induction with respect to $\beta$. The cases when $\beta=0$ or $\beta$ is a
limit ordinal are clear. Let $\beta=\mu+1$ and suppose that $T^*z \notin
d_\varepsilon^{\alpha+\beta}(B_{X^*})$. If $z \notin
d_\varepsilon^\mu(B_{Z^*})$, then the proof is finished. So we proceed assuming
that
 $z \in d_\varepsilon^\mu(B_{Z^*})$,
which by the inductive hypothesis implies that $T^*z \in
d_\varepsilon^{\alpha+\mu}(B_{X^*})$. There exist $x \in X$, $t>0$, such that
$T^*z \in H(x,t) \cap d_\varepsilon^{\alpha+\mu}(B_{X^*})=S$ and $\diam S
<\varepsilon$. Consider the slice $S'=H(Tx,t) \cap d_\varepsilon^\mu(B_{Z^*})$.
We have $\duality{Tx,z}=\duality{x,T^*z}$, so $z \in S'$.
 Also, $\diam S'\leq\diam S<\varepsilon$ as $T^*$ is an isometry.
We conclude that $z \notin d_\varepsilon^\beta(B_{Z^*})$, which finishes the
argument.
\end{proof}

Let us introduce  a shift operator $\tau_m:
\ell_1([0,\omega])\to\ell_1([0,\omega])$, $m\in\mathbb{N}$, by letting $\tau_m
h(n)=h(n-m)$ for $n\geq m$, $\tau_m h(n)=0$ for $n<m$ and $\tau_m
h(\omega)=h(\omega)$.

\begin{cor}
Let $h \in d^\alpha_\varepsilon(B_{\ell_1([0,\omega])})$. Then $\tau_m h \in
d^\alpha_\varepsilon(B_{\ell_1([0,\omega])})$ for every $m \in \Natural$.
\end{cor}
\begin{proof}
Indeed, consider the mapping $T:C([0,\omega]) \to C([0,\omega])$ defined  as

$T((x(0),x(1),\ldots,x(\omega)))=(x(1),x(2),\ldots,x(\omega))$. Clearly,
$T^*=\tau_1$ and the assertion for $m=1$ follows by the previous proposition.
For $m>1$ one may use induction.
\end{proof}

\begin{defn}
Let $\alpha$ be an ordinal and $\varepsilon>0$. We will say that a subset $M$
of $X^*$ is an \emph{$\varepsilon$-$\alpha$-obstacle} for $f \in B_{X^*}$ if
\item{(i)} $\dist(f,M)\geq\varepsilon$,
\item{(ii)} for every $\beta < \alpha$ and every $w^*$-slice $S$ of $d_\varepsilon^\beta(B_{X^*})$ with $f \in S$ we have
$S \cap M \neq \emptyset$.
\end{defn}
It follows by transfinite induction that if $f$ has an
$\varepsilon$-$\alpha$-obstacle, then $f \in d_\varepsilon^\alpha(B_{X^*})$.

An \emph{$(n,\varepsilon)$-tree} in a Banach space $X$ is a finite
sequence $(x_i)_{i=0}^{2^{n+1}-1}\subset X$ such that
\[
x_i=\frac{x_{2i}+x_{2i+1}}{2} \mbox{ and } \norm{x_{2i}-x_{2i+1}}\geq
\varepsilon
\]
for $i=0,\ldots, 2^{n}-1$. The element $x_0$ is called the
\emph{root} of the tree $(x_i)_{i=0}^{2^{n+1}-1}$. Note that if
$(h_i)_{i=0}^{2^{n+1}-1}\subset B_{X^*}$ is an
$(n,\varepsilon)$-tree in $X^*$, then $h_0\in
d_\varepsilon^n(B_{X^*})$.

Define $f_\beta\in \ell_1([0,\alpha])$, for $\alpha \geq \beta$, by
$f_\beta(\xi)=1$ if $\xi=\beta$ and $f_\beta(\xi)=0$ otherwise.

\begin{lem}\label{l:omega}
\[f_\omega \in d^\omega_{1/2}(B_{\ell_1([0,\omega])})
\]
\end{lem}

\begin{proof}
In \cite[Exercise 9.20]{FHHMPZ}
 a sequence is constructed of $(n,1)$-trees in $B_{\ell_1([0,\omega])}$
 with roots
\[r_n=(\underbrace{\frac{1}{2^n},\ldots,\frac{1}{2^n}}_{2^n-times},0,\ldots)
\]
whose elements belong to ${\cal P}=\set{h\in
B_{\ell_1([0,\omega])}:\|h\|_1=1,\ h(n)\geq 0,\ h(\omega)=0}$. We
have $r_n \in d^{2n}_{1/2}(B_{\ell_1([0,\omega])})$, and
$\dist(f_\omega,{\cal P})=2$. Finally, for every $h \in {\cal P}$,
every $x \in C([0,\omega])$ and every $t \in \Real$ such that
$f_\omega \in H(x,t)$, there exists $m \in \Natural$ such that
$\tau_m h \in H(x,t)$. Therefore the set $\set{\tau_m r_n:(m,n) \in
\Natural^2}$ is an $\frac{1}{2}$-$\omega$-obstacle for $f_\omega$.
Thus $f_\omega \in d^\omega_{1/2}(B_{\ell_1([0,\omega])})$.
\end{proof}

\begin{prop}\label{prop:obstacle}
For every $\alpha < \omega$,
\begin{equation}\label{equation:nonempty}
f_{\omega^{\omega^\alpha}} \in
d_{1/2}^{\omega^{1+\alpha}}(B_{\ell_1([0,\omega^{\omega^\alpha}])})
\end{equation}
\end{prop}

\begin{proof}
The case $\alpha=0$ is contained in Lemma~\ref{l:omega}. Let us suppose that we
have proved the assertion~\eqref{equation:nonempty} for all ordinals (natural
numbers, in fact) less than or equal to $\alpha$. It is enough to show, for
every $n \in \Natural$, that
\begin{equation}\label{equation:n-times}
f_{\left(\omega^{\omega^\alpha}\right)^n} \in d_{1/2}^{\omega^{1+\alpha}n}
(B_{\ell_1([0,\left(\omega^{\omega^\alpha}\right)^n])}).
\end{equation}
Indeed, \eqref{equation:n-times} implies
\[
f_{\left(\omega^{\omega^\alpha}\right)^n} \in
d_{1/2}^{\omega^{1+\alpha}n}(B_{\ell_1([0,\omega^{\omega^{\alpha+1}}])}).
\]
Since $f_{(\omega^{\omega^\alpha})^n} \stackrel{w^*}{\longrightarrow}
f_{\omega^{\omega^{\alpha+1}}}$ and
$\norm{f_{(\omega^{\omega^\alpha})^n}-f_{\omega^{\omega^{\alpha+1}}}}=2$, we
see that $\{f_{\left(\omega^{\omega^\alpha}\right)^n}: n\in \Natural\}$ is an
$\frac{1}{2}$-$\omega^{1+\alpha+1}$-obstacle for
$f_{\omega^{\omega^{\alpha+1}}}$. That implies~\eqref{equation:nonempty} for
$\alpha+1$.

In order to prove \eqref{equation:n-times} we will proceed by induction. The
case $n=1$ follows from the inductive hypothesis as indicated above, so let us
suppose that $n=m+1$ and $\eqref{equation:n-times}$ holds for $m$.

\noindent  Define the mapping $T: C([0,(\omega^{\omega^\alpha})^n]) \rightarrow
C([0,\omega^{\omega^\alpha}])$ by
\[
\begin{array}{rl}
(Tx)(\gamma)=x((\omega^{\omega^\alpha})^m(1+\gamma)),\, \gamma \leq
\omega^{\omega^\alpha}
\end{array}
\]
A simple computation shows that the dual map $T^*$ is given by
\[
\begin{array}{rl}(T^*g)(\gamma)=
\begin{cases}
g(\xi), \mbox{ if } \gamma=(\omega^{\omega^\alpha})^m(1+\xi),\, \xi\leq \omega^{\omega^\alpha}\\
0 \mbox{ otherwise}
\end{cases}
\end{array}
\]
Clearly, $T^*$ is an isometric isomorphism of
$\ell_1([0,\omega^{\omega^\alpha}])$ onto $\rng T^*$. We claim that

\begin{equation}\label{equation:suba}B_{\ell_1([0,(\omega^{\omega^\alpha})^n])}\cap \rng T^*
\subset d_{1/2}^{\omega^{1+\alpha}m}
(B_{\ell_1([0,(\omega^{\omega^\alpha})^n])}).
\end{equation}

Note that the set of extremal points of
$B_{\ell_1([0,(\omega^{\omega^\alpha})^n])}\cap\rng T^*$ satisfies

$$
\text{ext}(B_{\ell_1([0,(\omega^{\omega^\alpha})^n])}\cap\rng
T^*)\subset\{f_\gamma,-f_\gamma:
\gamma=(\omega^{\omega^\alpha})^m(1+\xi),\, \xi\leq
\omega^{\omega^\alpha}\}
$$
By the inductive assumption and by symmetry,
$f_{(\omega^{\omega^\alpha})^m}$ and
$-f_{(\omega^{\omega^\alpha})^m}$ belong to
$d_{1/2}^{\omega^{1+\alpha}m}
(B_{\ell_1([0,(\omega^{\omega^\alpha})^n])})$. It is easy to see
that more generally, $f_\gamma$ and $-f_\gamma$ belong to
$d_{1/2}^{\omega^{1+\alpha}m}
(B_{\ell_1([0,(\omega^{\omega^\alpha})^n])})$, whenever
$\gamma=(\omega^{\omega^\alpha})^m(1+\xi),\, \xi\leq
\omega^{\omega^\alpha}$. Thus we have verified that

\[\text{ext}(B_{\ell_1([0,(\omega^{\omega^\alpha})^n])}
\cap \rng T^* )\subset d_{1/2}^{\omega^{1+\alpha}m}
(B_{\ell_1([0,(\omega^{\omega^\alpha})^n])}),\] and the claim
~\eqref{equation:suba} follows using the Krein-Milman theorem.

This together with the inductive assumption
\eqref{equation:nonempty} allows us to
 apply Proposition~\ref{proposition:fundamental}
(with $\ell_1([0,(\omega^{\omega^\alpha})^n])$ as $X^*$,
$C([0,\omega^{\omega^\alpha}])$ as $Z$, and $\rng T^*$ as $Y$) to get
\[f_{(\omega^{\omega^\alpha})^n}=T^*f_{\omega^{\omega^\alpha}} \in d_{1/2}^{\omega^{1+\alpha}n}
(B_{\ell_1([0,(\omega^{\omega^\alpha})^n])}).\]
\end{proof}

To finish the proof of Theorem~\ref{t:main}, we use that for every Asplund
space $X$,
 $\dent(X)=\omega^\xi$ for some ordinal $\xi$
(see~\cite[Proposition 3.3]{L}, \cite{HMVZ}). Combining
Proposition~\ref{prop:obstacle} with \eqref{i:inequality1} we obtain
\[\dent(C([0,\omega^{\omega^\alpha}]))= \omega^{1+\alpha+1}\] for $\alpha < \omega$. For
$\omega\le\alpha<\omega_1$, we use that
$\omega^{1+\alpha+1}=\omega^{\alpha+1}=\szlenk(C([0,\omega^{\omega^\alpha}])) =
\dent(C([0,\omega^{\omega^\alpha}]))$, which finishes the proof.
\end{proof}

Our next proposition is a direct consequence of Theorem
\ref{t:main}, Lemma \ref{l:szlenkL2} and Proposition
\ref{p:upperszlenkL2}.

\begin{prop} Let $0\le \alpha <\omega_1$. Then
$\szlenk(L_2(C([0,\omega^{\omega^\alpha}])))=\omega^{1+\alpha+1}$.
\end{prop}

Our main result can be extended to the non separable case as
follows.

\begin{thm} Let $0\le \alpha <\omega_1$. Let $K$ be a compact space
whose Cantor derived sets satisfy $K^{\omega^\alpha} \neq \emptyset$
and $K^{\omega^{\alpha+1}}=\emptyset$. Then
$\dent(C(K))=\omega^{1+\alpha+1}$.
\end{thm}

\begin{proof} The upper estimate follows from the separable
determination of the weak$^*$-dentability index when it is countable
and from Theorem \ref{t:main} (the argument is identical to the one
given for the computation of Sz$(C(K))$ in \cite{La2}).

On the other hand, since $K^{\omega^\alpha} \neq \emptyset$, we have
that $\szlenk(C(K))\ge \omega^{\alpha+1}$ (see \cite{La2} or
Proposition 7 in \cite{L}). Therefore there is a separable subspace
$X$ of $C(K)$ such that $\szlenk(X)\ge \omega^{\alpha+1}$. By
considering the closed subalgebra of $C(K)$ generated by $X$, we may
as well assume that $X$ is isometric to $C(L)$, where $L$ is a
compact metrizable space. Since $\szlenk(C(L))\ge
\omega^{\alpha+1}$, it follows from Theorem \ref{t:main} that
$\dent(C(L))\ge \omega^{1+\alpha+1}$ and finally that
$\dent(C(K))\ge \omega^{1+\alpha+1}$.
\end{proof}

\end{document}